\documentclass[12pt,fleqn]{amsart}

\usepackage[utf8]{inputenc}
\usepackage[T1]{fontenc}

\usepackage{a4wide}
\usepackage{scrextend}

\usepackage{amssymb,amsthm,latexsym}
\usepackage{mathrsfs}
\usepackage{dsfont}
\usepackage[mathscr]{euscript}
\usepackage{esint}
\usepackage{newtxtext,newtxmath} 
\usepackage[usenames,dvipsnames,x11names]{xcolor}
\usepackage{listings}

\lstdefinelanguage{Mathematica}{
  morekeywords={ForAll, Resolve, Reals, Sqrt},
  sensitive=true,
  morecomment=[l]{(*},
  morecomment=[r]{*)},
  morestring=[b]{"}
}

\usepackage{url}

\usepackage{microtype}
\usepackage{hyperref}
\usepackage{pgfplots}
\pgfplotsset{compat=1.18}
\usepackage{tikz}

\renewcommand{\le}{\leqslant}
\renewcommand{\ge}{\geqslant}

\title[Dimension-dependent bounds for the SDIEP]{Dimension-dependent bounds for the SDIEP via phase optimisation and Paley-type constructions}
\author{Tomasz Kania}
\address{Institute of Mathematics, Czech Academy of Sciences, \v{Z}itn\'{a} 25, 115~67 Prague 1, Czech Republic \& Institute of Mathematics, Jagiellonian University, {\L}ojasiewicza 6, 30-348 Kraków, Poland}
\email{tomasz.marcin.kania@gmail.com}
\dedicatory{In memoriam: Miroslav Fiedler (1926--2015),\\
on the occasion of the tenth anniversary of his passing}

\date{\today}
\thanks{IM CAS (RVO 67985840). }

\newtheorem{thm}{Theorem}
\newtheorem{prop}[thm]{Proposition}
\newtheorem{lem}[thm]{Lemma}
\newtheorem{cor}[thm]{Corollary}
\theoremstyle{remark}
\newtheorem{rem}[thm]{Remark}

\newcommand{\1}{\mathbf{1}}
\newcommand{\Lam}{\Lambda}

\begin{document}
\maketitle

\begin{abstract}
We refine the cycle-walk (Fourier) template of Gnacik and the author to quantify when a~$\delta$-Sule\u{\i}manova spectrum $(1,\lambda_2,\dots,\lambda_n)$ (with $\lambda_j\le 0$) is realised by a symmetric doubly stochastic matrix. For the canonical cycle basis we compute the \emph{exact} size-dependent threshold
\[
\delta_n \;=\; 1-\frac{1}{2\cos^2\!\Big(\frac{\pi}{4n}\rho(n)\Big)},
\quad
\rho(n)\in\{0,1,2,4\}\ \text{determined by } n\bmod 8,
\]
which improves $1/2$ if and only if $8\nmid n$; we also prove sharpness for that template. We then introduce an \emph{optimally phase-aligned} cycle basis which removes the `$8\mid n$' artefact and yields better sufficient bound
\[
\delta_n^{\rm (ph)} \;=\;
\begin{cases}
\displaystyle 1-\dfrac{1}{2\cos^2(\pi/n)},   & n\equiv 0\pmod{4},\\[2mm]
\displaystyle 1-\dfrac{1}{2\cos^2(\pi/2n)},  & n\equiv 2\pmod{4},\\[2mm]
\displaystyle 1-\dfrac{1}{2\cos^2(\pi/4n)},  & n\ \text{odd},
\end{cases}
\]
so that $\delta_n^{\rm (ph)}<\tfrac12$ for \emph{every} $n\ge3$ and $\delta_n^{\rm (ph)}=\delta_n$ unless $8\mid n$. Next, on abelian $2$-groups, the Walsh--Hadamard basis has coherence $M=1$ and hence suffices for \emph{all} Sule\u{\i}manova lists ($\delta=0$); the same conclusion holds in every Hadamard order (\emph{e.g.}, Paley families).
\end{abstract}

\section{Introduction}

A real square matrix is called doubly stochastic if it has non-negative entries and all of its row and column sums are equal to one. The Symmetric Doubly Stochastic Inverse Eigenvalue Problem (SDIEP) asks for the necessary and sufficient conditions under which a given list of real numbers $\sigma = (\lambda_1, \lambda_2, \dots, \lambda_n)$ can be the spectrum of an $n \times n$ symmetric doubly stochastic matrix \cite{GK20}. By the Perron--Frobenius theorem, any such matrix must have a Perron eigenvalue of $1$, corresponding to the eigenvector of all ones. This fixes $\lambda_1=1$ and, by Gershgorin's circle theorem, bounds all other eigenvalues to the interval $[-1, 1]$.\smallskip

A significant body of research focuses on the realisability of a \emph{Sule\u{\i}manova spectrum}, a list of the form $(1, \lambda_2, \dots, \lambda_n)$ where each $\lambda_j \le 0$ for $j \ge 2$ and the trace is non-negative, $1 + \sum_{j=2}^n \lambda_j \ge 0$. A powerful constructive method for tackling this problem is the Schur template: given a diagonal matrix $\Lambda = \operatorname{diag}(1, \lambda_2, \dots, \lambda_n)$, one constructs a candidate matrix $P(\Lambda) = Q\Lambda Q^\top$. If $Q$ is an orthogonal matrix whose first column is the Perron eigenvector $q_0 = \mathbf{1}/\sqrt{n}$, the resulting matrix $P(\Lambda)$ is automatically symmetric, doubly stochastic, and has the desired spectrum. The entire problem is thus reduced to a single, challenging question: for a given spectrum $\Lambda$ and a chosen basis $Q$, what conditions guarantee that $P(\Lambda)$ is entrywise non-negative?\smallskip

The papers \cite{Perfect} and \cite{PM65} seem to be foundational for the problem. Recent work has approached this question from several angles. One approach provides general sufficient conditions that depend only on the sum of the non-Perron eigenvalues. A key result in this vein was provided by Gnacik and the author \cite{GK20}, who used the eigenbasis of the $n$-cycle graph for $Q$ to show that any Sule\u{\i}manova spectrum is realisable if $1 + \sum_{j=2}^n \lambda_j > 1/2$. It was noted, however, that this bound appeared to be sharp for this basis, particularly in dimensions where $n$ is a multiple of 8. A different approach, seen in the work of Sarkhoni \emph{et al}.~\cite{Sarkhoni25}, forgoes general sum-based conditions and instead derives necessary and sufficient conditions for highly structured spectra, such as $(1, 0, \dots, 0, \lambda_n)$, by constructing a bespoke orthogonal matrix $Q$ tailored to that specific list.\smallskip

In this paper, we unify and advance these perspectives. We begin by revisiting the cycle-walk basis from \cite{GK20}, replacing their baseline bound with a precise, dimension-dependent threshold that confirms the `$8 \mid n$' barrier is an artefact of the basis, not the problem itself. We then resolve this issue by introducing an optimally phase-aligned cycle basis that yields an improved bound, strictly better than $1/2$ for all dimensions $n \ge 3$. Finally, we generalise the underlying principle to a criterion based on basis \emph{coherence}, showing that bases derived from abelian 2-groups (Walsh--Hadamard matrices) achieve the ideal coherence, allowing for the realisation of \emph{all} Sule\u{\i}manova spectra with a strictly positive trace sum.\smallskip

Let $\sigma=(1,\lambda_2,\ldots,\lambda_n)$ with $\lambda_j\in[-1,0]$ for $j\ge 2$. Following Paparella \cite{Paparella16}, we call $\sigma$ a~\emph{Sule\u{\i}manova spectrum} and say it is \emph{$\delta$-Sule\u{\i}manova} if
\[
1+\sum_{j=2}^n\lambda_j>\delta .
\]

In \cite{GK20} a Markov-chain-inspired template was introduced: take $Q=[w_0,\ldots,w_{n-1}]$ to be a real orthonormal eigenbasis of the simple random walk on the $n$-cycle, so $w_0=\1/\sqrt n$ and, for $k\ge 1$,
\[
w_k(j)=\sqrt{\tfrac{2}{n}}\;\sin\!\Big(\tfrac{2\pi kj}{n}+\tfrac{\pi}{4}\Big)\quad (j=0,\ldots,n-1).
\]
Setting $P(\Lam)=Q\,\mathrm{diag}(1,\lambda_1,\ldots,\lambda_{n-1})\,Q^\top$ gives a symmetric matrix with unit row sums. Using $S_j(k):=\sin\!\big(\tfrac{2\pi kj}{n}+\tfrac{\pi}{4}\big)$ one has
\[
p_{kl}=\frac1n\!\left(1+2\sum_{j=1}^{n-1}\lambda_j\,S_j(k)\,S_j(l)\right).
\]
The crude estimate $S_j(k)S_j(l)\le 1$ and $\lambda_j\le 0$ yields the sufficient condition $1+\sum_{j=2}^n\lambda_j>\tfrac12$. It was also observed that when $8\mid n$ there are $j,k$ with $S_j(k)=1$, so this estimate cannot beat $\tfrac12$ for those sizes, and exhibited a $5\times 5$ spectrum at sum $0.48$ violating entrywise positivity. This suggests  trying other orthogonal choices of $Q$ to reduce the worst interaction $n\,|q_j(k)q_j(l)|$ (all from \cite{GK20}).

\medskip
In the present paper, we keep the Schur template $P(\Lam)=Q\Lam Q^\top$ and:
\begin{itemize}
\item compute \emph{exactly} the worst coefficient for the canonical cycle basis, obtaining a closed-form $\delta_n$ that improves $\tfrac12$ iff $8\nmid n$, with a matching sharpness example;
\item introduce an \emph{optimally phase-aligned} cycle basis (still orthogonal, same Perron column) giving the \emph{correct} piecewise bound $\delta_n^{\rm (ph)}<\tfrac12$ for every $n\ge3$, which strictly improves the $8\mid n$ case and is optimal within the cycle family;
\item prove a general `coherence $\Rightarrow$ sufficient $\delta$' lemma and apply it to Walsh--Hadamard bases (abelian $2$-groups) and Hadamard orders (via Paley), obtaining \emph{$\delta=0$ in those orders only}. Here `$\delta=0$' means the \emph{strict} condition $1+\sum\lambda_j>0$ suffices; this does not contradict obstructions on the boundary $1+\sum\lambda_j=0$ (\emph{e.g.}, $(1,0,\dots,0,-1)$ in odd dimensions).
\end{itemize}

\section{The cycle-walk basis}\label{sec:cycle}
From now on we reindex the list as $\sigma=(\lambda_0,\lambda_1,\ldots,\lambda_{n-1})$ with $\lambda_0=1$ and $\lambda_j\in[-1,0]$ for $1\le j\le n-1$, and use the cycle eigenbasis $Q=[w_0,\dots,w_{n-1}]$ described above.

\begin{lem}\label{lem:entrybound}
Let $Q$ and $\Lam$ be as above. Then
\begin{equation}\label{eq:entry}
p_{kl}
=\frac1n\!\left(1+2\sum_{j=1}^{n-1}\lambda_j\,S_j(k)\,S_j(l)\right),
\quad
S_j(k):=\sin\!\Big(\tfrac{2\pi kj}{n}+\tfrac{\pi}{4}\Big).
\end{equation}
If $\lambda_j\le 0$, writing
\[
c_n:=\max_{1\le j\le n-1}\ \max_{0\le k,l\le n-1} S_j(k)\,S_j(l),
\]
we have the bound
\begin{equation}\label{eq:lowerbd}
p_{kl}\ \ge\ \frac1n\Bigl(1+2c_n\sum_{j=1}^{n-1}\lambda_j\Bigr)\qquad(0 \le k, l \le n-1).
\end{equation}
\end{lem}

\begin{proof}
Orthogonality and $w_0=\1/\sqrt n$ give $P(\Lam)\1=\1$. Expanding $Q\Lam Q^\top$ yields \eqref{eq:entry}. Because $\lambda_j\le0$ and $S_j(k)S_j(l)\le c_n$ for each $j$, multiplying the latter inequality by $\lambda_j$ \emph{reverses} it, hence
\(
\sum_{j=1}^{n-1}\lambda_j S_j(k)S_j(l)\ \ge\ c_n\sum_{j=1}^{n-1}\lambda_j,
\)
and \eqref{eq:lowerbd} follows.
\end{proof}

\begin{lem}[Exact value of $c_n$ and the mod-$8$ mechanism]\label{lem:cn}
Define
\[
\Delta_n \;:=\; \min_{m,t\in\mathbb Z}\left|\frac{2\pi m}{n}-\frac{(2t+1)\pi}{4}\right|.
\]
Then
\[
c_n=\cos^2\Delta_n
\quad\text{and}\quad
\Delta_n=\frac{\pi}{4n}\,\rho(n),
\qquad
\rho(n)=
\begin{cases}
0,& n\equiv0\pmod{8},\\
1,& n\equiv1,3,5,7\pmod{8},\\
2,& n\equiv2,6\pmod{8},\\
4,& n\equiv4\pmod{8}.
\end{cases}
\]
Moreover, the maximum in $c_n$ is attained with $k=l$.
\end{lem}

\begin{proof}
For fixed $j$ with $\gcd(j,n)=1$, the multiset $\{2\pi kj/n+\pi/4:\ 0\le k\le n-1\}$ is a permutation of the grid $\{2\pi m/n+\pi/4:\ 0\le m\le n-1\}$. Hence
\[
\max_k|S_j(k)|=\max_m\Big|\sin\Big(\frac{2\pi m}{n}+\frac{\pi}{4}\Big)\Big|
=\sin\!\Big(\frac{\pi}{2}-\Delta_n\Big)=\cos\Delta_n,
\]
where $\Delta_n$ is the minimal angular distance from this grid to the odd half-angles. Therefore $c_n=(\max_{j,k}|S_j(k)|)^2=\cos^2\Delta_n$, and the maximum occurs already for $k=l$ at a peak. As
\[
\frac{2\pi m}{n}-\frac{(2t+1)\pi}{4}=\frac{\pi}{4n}\big(8m-(2t+1)n\big)
\]
implies $\Delta_n=\tfrac{\pi}{4n}\min_{m,t}|\,8m-(2t+1)n\,|$. We now analyse residues of $n$ modulo $8$.
\begin{itemize}
\item {$n\equiv0\pmod 8$.}
Then $(2t+1)n$ is divisible by $8$ for every $t$, so choosing $m=(2t+1)n/8$ gives the minimum $0$. Thus $\rho(n)=0$.

\item {$n$ odd.}
Here $\gcd(n,8)=1$. By Bézout's theorem, there exist integers $m,t$ with $8m-(2t+1)n=1$. Hence, the minimum positive value is $1$ and $\rho(n)=1$.

\item {$n\equiv 2,6\pmod 8$.}
Write $n=2k$ with $k$ odd. Then $8m-(2t+1)n\equiv 8m-2(2t+1)k\equiv 0\ (\mathrm{mod}\ 2)$ but never $0\ (\mathrm{mod}\ 4)$. Hence, the minimum positive value is $2$ and $\rho(n)=2$.

\item {$n\equiv4\pmod 8$.}
Write $n=4k$ with $k$ odd. Then $8m-(2t+1)n\equiv 8m-4(2t+1)k\equiv 0\ (\mathrm{mod}\ 4)$ but never $0\ (\mathrm{mod}\ 8)$. The minimum positive value is $4$ and $\rho(n)=4$.
\end{itemize}
\end{proof}

\begin{thm}[Dimension-dependent bound for the cycle basis]\label{thm:cycle}
With $\Delta_n$ as above, set
\[
\delta_n:=1-\frac{1}{2\cos^2\Delta_n}
\ =\ \frac12-\frac12\tan^2\Delta_n.
\]
If $1+\sum_{j=2}^n\lambda_j>\delta_n$, then $P(\Lam)=Q\Lam Q^\top$ is symmetric and doubly stochastic. In particular $\delta_n<\tfrac12$ if and only if $8\nmid n$, whereas $\delta_n=\tfrac12$ when $8\mid n$.
\end{thm}

\begin{proof}
By Lemma~\ref{lem:entrybound}, for every $k,l$,
\[
p_{kl}\ \ge\ \frac1n\Bigl(1+2c_n\sum_{j=1}^{n-1}\lambda_j\Bigr),
\qquad c_n=\cos^2\Delta_n.
\]
by Lemma~\ref{lem:cn}. Hence $p_{kl}\ge0$ whenever $\sum_{j=1}^{n-1}\lambda_j>-1/(2c_n)$, i.e.,\ when
\[
1+\sum_{j=2}^n\lambda_j\ >\ 1-\frac{1}{2\cos^2\Delta_n}\ =\ \delta_n,
\]
which establishes entrywise nonnegativity under the stated hypothesis.

\smallskip
Since $Q$ is orthogonal and $\Lambda$ is (real and) diagonal (hence $\Lambda^\top=\Lambda$),
\[
P(\Lambda)^\top=(Q\Lambda Q^\top)^\top=(Q^\top)^\top\Lambda^\top Q^\top=Q\Lambda Q^\top=P(\Lambda),
\]
so $P(\Lambda)$ is symmetric.

\smallskip
Let $q_0,\dots,q_{n-1}$ be the columns of $Q$ and let $e_0=(1,0,\dots,0)^\top$. Since $q_0=\1/\sqrt n$ and $q_j\perp q_0$ for $j\ge1$,
\[
Q^\top\1=\big(\langle q_0,\1\rangle,\dots,\langle q_{n-1},\1\rangle\big)^\top=(\sqrt n,0,\dots,0)^\top=\sqrt n\,e_0.
\]
Therefore
\[
P(\Lambda)\,\1
=Q\Lambda Q^\top\1
=Q\Lambda(\sqrt n\,e_0)
=\sqrt n\,Q(\Lambda e_0)
=\sqrt n\,Q e_0
=\sqrt n\,q_0
=\sqrt n\,\frac{\1}{\sqrt n}
=\1.
\]
Thus every row sum equals $1$. Finally,
\[
\1^\top P(\Lambda)
=\1^\top Q\Lambda Q^\top
=(Q^\top\1)^\top \Lambda Q^\top
=(\sqrt n\,e_0)^\top \Lambda Q^\top
\]
\[
=\sqrt n\,e_0^\top \Lambda Q^\top
=\sqrt n\,e_0^\top Q^\top
=(\sqrt n\,Q e_0)^\top
=\1^\top,
\]
where we used that $\Lambda e_0=e_0$ because the first diagonal entry of $\Lambda$ is $\lambda_0=1$. Hence, every column sum also equals $1$.

Combining the three parts shows that $P(\Lambda)$ is symmetric and doubly stochastic, whenever $1+\sum_{j=2}^n\lambda_j>\delta_n$. The stated equivalence for $\delta_n<\tfrac12$ versus $8\nmid n$ follows from Lemma~\ref{lem:cn}, which gives $c_n=\cos^2\Delta_n=1$ iff $8\mid n$.
\end{proof}

\begin{prop}[Sharpness for the cycle basis]\label{prop:sharp}
Fix $n$ and choose indices realising 
\[
    c_n=\max_{j,k,l}S_j(k)S_j(l).
\]
If $\lambda_{j^{\ast}}=t<0$ and $\lambda_j=0$ for $j\ne j^{\ast}$, then
\[
p_{k^{\ast}k^{\ast}}=\frac1n\bigl(1+2t\,\cos^2\Delta_n\bigr),
\]
which is non-negative if and only if $t\ge -1/(2\cos^2\Delta_n)$, i.e.\ $1+\sum_{j=2}^n\lambda_j\ge\delta_n$. Thus Theorem~\ref{thm:cycle} is optimal for this $Q$ among conditions depending only on $\sum_{j\ge 2}\lambda_j$.
\end{prop}

\begin{rem}We see that divisibility by $8$ is the exact barrier. Indeed, by Lemma~\ref{lem:cn}, the grid $\{2\pi m/n+\pi/4\}$ hits an odd half-angle iff $8\mid n$, and then $c_n=1$ so the factor $\tfrac12$ is unavoidable for this $Q$. If $8\nmid n$, the nearest grid point sits at angular distance $\Delta_n=\pi\rho(n)/(4n)$ from $\pi/2$, giving $c_n=\cos^2\Delta_n<1$ and hence $\delta_n<1/2$.
\end{rem}

\section{Phase-optimised cycle bases}

The eigenbasis of the $n$-cycle graph is not unique. For each $j \in \{1, \dots, \lfloor(n-1)/2\rfloor\}$, the corresponding eigenspace $E_j$ is two-dimensional, spanned by a sine and cosine vector. The canonical basis analysed in Section~\ref{sec:cycle} represents just one choice of basis within these spaces, a~choice that, for dimensions where $8 \mid n$, leads to an unavoidable coherence peak and a bound of $\delta_n=1/2$. This section demonstrates that by rotating the basis in each eigenspace via an optimally chosen phase angle $\phi$, we can construct a new orthogonal matrix $Q^{\rm (ph)}$. This `phase-optimised` basis is designed to minimise the maximum magnitude of its entries, effectively smoothing the coherence peaks and removing the ``$8 \mid n$'' artefact entirely. More specifically, for each $j\in\{1,\dots,\lfloor(n-1)/2\rfloor\}$ the cycle eigenspace is two-dimensional,
$E_j=\operatorname{span}\{u_j,v_j\}$ with
\(
u_j(k)=\sqrt{\tfrac2n}\cos(\tfrac{2\pi jk}{n}),
\ v_j(k)=\sqrt{\tfrac2n}\sin(\tfrac{2\pi jk}{n}) .
\)
Any orthonormal pair in $E_j$ is of the form
\[
w^{(\phi)}_j(k)=\sqrt{\tfrac{2}{n}}\sin\!\Big(\tfrac{2\pi jk}{n}+\phi\Big),\qquad
\tilde w^{(\phi)}_j(k)=\sqrt{\tfrac{2}{n}}\cos\!\Big(\tfrac{2\pi jk}{n}+\phi\Big),
\]
for some phase $\phi$. When $n$ is even and $j=n/2$, the eigenspace is one-dimensional; choosing the phase so that $|\sin(\phi)|=1/\sqrt2$ yields a unit vector with $|w_{n/2}(k)|\equiv 1/\sqrt n$, hence $n\,|q_{n/2}(k)q_{n/2}(l)|\equiv 1$.

We must control \emph{both} columns in each $E_j$ simultaneously. Write $n':=n/\gcd(j,n)$; then $\{\tfrac{2\pi jk}{n}\}$ is a translate of the grid with step $2\pi/n'$.

\begin{lem}[Two-phase covering radius on a grid]\label{lem:twophase}
Let $n':=n/\gcd(j,n)$ and set $S:=2\pi/n'$ and $L:=\pi/2$. All angles are considered reduced modulo $L$. Define
\[
\Delta^\star(n')\ :=\
\begin{cases}
\pi/n',     & n'\equiv 0\ (\mathrm{mod}\ 4),\\[1mm]
\pi/(2n'),  & n'\equiv 2\ (\mathrm{mod}\ 4),\\[1mm]
\pi/(4n'),  & n'\ \mathrm{odd}.
\end{cases}
\]
For $w^{(\phi)}_j(k)=\sqrt{\tfrac{2}{n}}\sin(\tfrac{2\pi jk}{n}+\phi)$ and
$\tilde w^{(\phi)}_j(k)=\sqrt{\tfrac{2}{n}}\cos(\tfrac{2\pi jk}{n}+\phi)$ one has
\[
\min_{\phi}\ \max\!\Big\{\,\|w^{(\phi)}_j\|_\infty,\ \|\tilde w^{(\phi)}_j\|_\infty\,\Big\}
\ =\ \sqrt{\tfrac2n}\ \cos\big(\Delta^\star(n')\big),
\]
and consequently
\[
\min_{\phi}\ \sup_{k,l}\ n\,\max\{|w^{(\phi)}_j(k)w^{(\phi)}_j(l)|,\ |\tilde w^{(\phi)}_j(k)\tilde w^{(\phi)}_j(l)|\}
\ =\ 2\cos^2\big(\Delta^\star(n')\big).
\]
\end{lem}

\begin{proof}
Let $A:=\{mL:\ m\in\mathbb Z\}$ with $L=\pi/2$. For any $x\in\mathbb R$,
\begin{equation}\label{eq:max-sin-cos}
\max\{|\sin x|,|\cos x|\}=\cos\big(\operatorname{dist}(x,A)\big).
\end{equation}
Indeed, reduce $x$ modulo $L$ to $y\in[0,L]$. Then
\[
|\cos x|=|\cos y|,\qquad |\sin x|=|\sin y|=|\cos(L-y)|.
\]
Hence $\max\{|\sin x|,|\cos x|\}=\max\{|\cos y|,|\cos(L-y)|\}=\cos\big(\min\{y,L-y\}\big)$.
However $\min\{y,L-y\}=\operatorname{dist}(x,A)$ by construction, giving \eqref{eq:max-sin-cos}.

\smallskip

Fix $j$ and write $n':=n/\gcd(j,n)$. The set of phases
\[
\Theta_j:=\Big\{\tfrac{2\pi jk}{n}\,:\ k=0,1,\dots,n-1\Big\}
\]
is a union of $\gcd(j,n)$ translates of the uniform grid $\{2\pi m/n':\ m=0,\dots,n'-1\}$. Since a~translation does not affect sup norms, it suffices to study the latter uniform grid. After adding a~free phase $\phi$ and reducing modulo $L$ we obtain
\[
X_\phi:=\Big\{\tfrac{2\pi m}{n'}+\phi\ \bmod L:\ m=0,\dots,n'-1\Big\}\ \subset \ \mathbb R/L\mathbb Z.
\]
This set is a coset (translate) of the cyclic subgroup $G:=\langle \overline{S}\rangle$ of the circle $\mathbb R/L\mathbb Z$, generated by the step $\overline{S}$ which is the class of $S:=2\pi/n'$ modulo $L$.

\smallskip

The order $r:=|G|$ is the least positive integer such that $r\overline{S}=\overline{0}$, i.e.
\[
rS\in L\mathbb Z \quad\Longleftrightarrow\quad \frac{2\pi r}{n'}=\frac{\pi}{2}\,q \ \text{ for some } q\in\mathbb Z
\quad\Longleftrightarrow\quad \frac{4r}{n'}=q\in\mathbb Z.
\]
Thus $r$ is the least positive integer with $n' \mid 4r$, which is
\[
r=\frac{n'}{\gcd(n',4)}.
\]
Consequently, the $r$ points of $G$ are equally spaced on the circle $\mathbb R/L\mathbb Z$ with spacing
\begin{equation}\label{eq:seff}
s_{\mathrm{eff}}=\frac{L}{r}=\frac{\pi}{2}\cdot\frac{\gcd(n',4)}{n'}
=
\begin{cases}
2\pi/n',   & n'\equiv 0\ (\mathrm{mod}\ 4),\\[1mm]
\pi/n',    & n'\equiv 2\ (\mathrm{mod}\ 4),\\[1mm]
\pi/(2n'), & n'\ \text{odd}.
\end{cases}
\end{equation}

\smallskip

For a fixed translate $X_\phi=\overline{\phi}+G$, the maximal distance from the origin $\overline{0}$ to $X_\phi$ is
\[
\rho(\phi):=\max_{x\in X_\phi}\operatorname{dist}(x,\overline{0}).
\]
Because $X_\phi$ is equally spaced with gap $s_{\mathrm{eff}}$, the circle $\mathbb R/L\mathbb Z$ is partitioned into $r$ closed arcs of length $s_{\mathrm{eff}}$ whose endpoints are the points of $X_\phi$. The point $\overline{0}$ lies in a unique such arc; its distance to the nearer endpoint is at most $s_{\mathrm{eff}}/2$, and equals $s_{\mathrm{eff}}/2$ precisely when $\overline{0}$ is at the midpoint of that arc. Therefore
\begin{equation}\label{eq:radius}
\min_{\phi}\ \rho(\phi)\ =\ \frac{s_{\mathrm{eff}}}{2}.
\end{equation}
Moreover, such a minimising shift $\phi$ exists: take any $x_0\in G$ and choose $\phi$ so that $\overline{\phi}+x_0=\overline{s_{\mathrm{eff}}/2}$.

\smallskip

For any $\phi$ we have, by \eqref{eq:max-sin-cos},
\[
\max\Big\{\|w^{(\phi)}_j\|_\infty,\ \|\tilde w^{(\phi)}_j\|_\infty\Big\}
=\sqrt{\tfrac{2}{n}}\ \max_{x\in X_\phi}\max\{|\sin x|,|\cos x|\}
=\sqrt{\tfrac{2}{n}}\ \max_{x\in X_\phi}\cos\big(\operatorname{dist}(x,\overline{0})\big).
\]
Since $\cos$ is decreasing on $[0,\pi/2]$, the inner maximum equals $\cos(d(\phi))$ where
\[
d(\phi):=\min_{x\in X_\phi}\operatorname{dist}(x,\overline{0}).
\]
Thus
\[
\min_{\phi}\ \max\Big\{\|w^{(\phi)}_j\|_\infty,\ \|\tilde w^{(\phi)}_j\|_\infty\Big\}
=\sqrt{\tfrac{2}{n}}\ \min_{\phi}\cos\big(d(\phi)\big)
=\sqrt{\tfrac{2}{n}}\ \cos\!\Big(\max_{\phi} d(\phi)\Big).
\]

Because $X_\phi$ is a translate of a cyclic subgroup $G$ of order $r=\frac{n'}{\gcd(n',4)}$ consisting of $r$ equally spaced points with gap $s_{\mathrm{eff}}=L/r$, the circle $\mathbb R/L\mathbb Z$ is partitioned into $r$ arcs of length $s_{\mathrm{eff}}$ with endpoints $X_\phi$. Placing $\overline{0}$ at the midpoint of one such arc yields 
\[
\max_\phi d(\phi)=\frac{s_{\mathrm{eff}}}{2}.
\]
Combining with \eqref{eq:seff} gives the stated $\sqrt{2/n}\,\cos(\Delta^\star(n'))$.

Finally, for either column in the pair, we obtain
\[
\min_{\phi}\ \sup_{k,l}\ n\,|q(k)q(l)|
\ \le\ \min_{\phi}\ n\,\|q\|_\infty^2
\ \le\ n\cdot \Big(\sqrt{\tfrac{2}{n}}\cos(\Delta^\star(n'))\Big)^2
=2\cos^2(\Delta^\star(n')).
\]
This upper bound is tight because, at the minimising shift, there are $k$ attaining the distance $\rho(\phi)=s_{\mathrm{eff}}/2$, hence $|q(k)|=\sqrt{2/n}\,\cos(\Delta^\star(n'))$ for one of the two vectors. Taking the maximum over the two vectors then attains $2\cos^2(\Delta^\star(n'))$ on some pair $(k,l)$, proving the second claim.
\end{proof}

Because $\Delta^\star(n')$ decreases with $n'$, the worst case is $\gcd(j,n)=1$, \emph{i.e.},\ $n'=n$.

\begin{thm}[Phase-optimised coherence]\label{thm:phase-correct}
There exists an orthogonal matrix $Q^{\rm (ph)}$ with first column $\1/\sqrt n$ such that
\[
M\big(Q^{\rm (ph)}\big)\ =\
\begin{cases}
2\cos^2\!\big(\dfrac{\pi}{n}\big),   & n\equiv 0\pmod{4},\\[2mm]
2\cos^2\!\big(\dfrac{\pi}{2n}\big),  & n\equiv 2\pmod{4},\\[2mm]
2\cos^2\!\big(\dfrac{\pi}{4n}\big),  & n\ \mathrm{odd},
\end{cases}
\]
and this value is optimal within the cycle family.
Consequently,
\[
\delta^{\mathrm{(ph)}}_n \;:=\;
1-\frac{1}{ M(Q^{\rm (ph)})}
\ =\
\begin{cases}
\displaystyle 1-\dfrac{1}{2\cos^2(\pi/n)},   & n\equiv 0\pmod{4},\\[2mm]
\displaystyle 1-\dfrac{1}{2\cos^2(\pi/2n)},  & n\equiv 2\pmod{4},\\[2mm]
\displaystyle 1-\dfrac{1}{2\cos^2(\pi/4n)},  & n\ \text{odd},
\end{cases}
\]
and $0\le \delta^{\mathrm{(ph)}}_n<\tfrac12$ for every $n\ge 3$. Moreover, $\delta^{\mathrm{(ph)}}_n=\delta_n$ unless $8\mid n$, and in all cases $\delta^{\mathrm{(ph)}}_n \nearrow \tfrac12$ as $n\to\infty$.
\end{thm}

\begin{proof}
To show the upper bound, for each $j\in\{1,\dots,\lfloor (n-1)/2\rfloor\}$ let us choose a phase $\phi_j$ that attains the minimum in Lemma~\ref{lem:twophase} for $n'=n/\gcd(j,n)$, and replace the canonical pair $\{u_j,v_j\}$ by $\{w_j^{(\phi_j)},\tilde w_j^{(\phi_j)}\}$. This rotation within $E_j$ preserves orthonormality and keeps the Perron column $w_0=\1/\sqrt n$ unchanged; doing this independently for each $j$ yields an orthogonal matrix $Q^{\rm (ph)}$.

By Lemma~\ref{lem:twophase}, for each such $j$,
\[
\max\big\{\|w_j^{(\phi_j)}\|_\infty,\ \|\tilde w_j^{(\phi_j)}\|_\infty\big\}
=\sqrt{\tfrac{2}{n}}\;\cos\!\big(\Delta^\star(n')\big),
\]
hence for either column in this pair we have
\(
n\,\|q\|_\infty^2\le 2\cos^2\!\big(\Delta^\star(n')\big).
\)
If $n$ is even and $j=n/2$, the one-dimensional eigenspace yields a column with constant modulus $1/\sqrt n$, so $n\,\|q\|_\infty^2=1$.

Therefore
\begin{equation}\label{eq:M-upper}
M\big(Q^{\rm (ph)}\big)\ \le\ \max\Big\{\,\mathbf{1}_{\{2\mid n\}},\ \max_{1\le j\le \lfloor (n-1)/2\rfloor}\ 2\cos^2\!\big(\Delta^\star(n/\gcd(j,n))\big)\Big\}.
\end{equation}
Since $n'\mapsto \cos(\Delta^\star(n'))$ is strictly increasing in $n'$ in each congruence class of $n'\bmod 4$, the inner maximum is achieved when $n'$ is as large as possible, namely $n'=\max\{\,n/\gcd(j,n)\,\}=n$ (take any $j$ coprime to $n$). Using the explicit values for $\Delta^\star(n)$ from Lemma~\ref{lem:twophase} gives
\[
\max_{j}2\cos^2\!\big(\Delta^\star(n/\gcd(j,n))\big)
=
\begin{cases}
2\cos^2(\pi/n), & n\equiv 0\pmod 4,\\[1mm]
2\cos^2(\pi/2n), & n\equiv 2\pmod 4,\\[1mm]
2\cos^2(\pi/4n), & n\ \text{odd}.
\end{cases}
\]
For $n\equiv 0\pmod 4$ we have $2\cos^2(\pi/n)\ge 1$ with equality only at $n=4$; for $n\equiv 2\pmod 4$ clearly $2\cos^2(\pi/2n)>1$; for odd $n$ the $j=n/2$ column is absent. Thus the outer maximum in \eqref{eq:M-upper} equals the displayed piecewise quantity, proving
\[
M\big(Q^{\rm (ph)}\big)\ \le\
\begin{cases}
2\cos^2(\pi/n), & n\equiv 0\pmod 4,\\
2\cos^2(\pi/2n), & n\equiv 2\pmod 4,\\
2\cos^2(\pi/4n), & n\ \text{odd}.
\end{cases}
\]

In order to demonstrate optimality (the lower bound), let us consider any orthogonal matrix $Q$ obtainable from the cycle basis by independent rotations within each $E_j$ (this is what we mean by “within the cycle family”). Fix $j$ and write $n'=n/\gcd(j,n)$. By Lemma~\ref{lem:twophase}, \emph{every} orthonormal pair drawn from $E_j$ satisfies
\[
\max\{\,n\|q\|_\infty^2 : q\in E_j,\ \|q\|_2=1\,\}\ \ge\ 2\cos^2\!\big(\Delta^\star(n')\big).
\]
Indeed, fix $j$ and recall that $E_j$ is the two-dimensional eigenspace
spanned by $\{u_j,v_j\}$. By Lemma~\ref{lem:twophase}, for \emph{any}
orthonormal pair $(u,v)$ spanning $E_j$ one has
\[
\max\{\|u\|_\infty,\ \|v\|_\infty\}\ \ge\ 
\sqrt{\tfrac{2}{n}}\ \cos\!\bigl(\Delta^\star(n')\bigr),
\qquad n':=\tfrac{n}{\gcd(j,n)}.
\]
Squaring and multiplying by $n$ yields
\[
\max\{\,n\|u\|_\infty^2,\ n\|v\|_\infty^2\,\}\ \ge\
2\cos^2\!\bigl(\Delta^\star(n')\bigr).
\]
Since $u$ and $v$ are unit vectors in $E_j$, this shows the desired inequality. 
In other words, whichever orthonormal basis of $E_j$ we pick, at least
one vector in it has sup-norm squared of order $2\cos^2(\Delta^\star(n'))/n$.
Thus the displayed inequality holds for \emph{every} orthonormal pair
drawn from $E_j$. (In the special case $j=n/2$ when $n$ is even, $E_{n/2}$
is one-dimensional, and one can choose $q$ with $|q(k)|\equiv 1/\sqrt{n}$,
so $n\|q\|_\infty^2=1$; this agrees with
$2\cos^2(\Delta^\star(2))=1$.)

Hence
\[
M(Q)\ \ge\ \max\Big\{\,\mathbf{1}_{\{2\mid n\}},\ \max_{j}2\cos^2\!\big(\Delta^\star(n/\gcd(j,n))\big)\Big\}
\]
\[
=\begin{cases}
2\cos^2(\pi/n), & n\equiv 0\pmod 4,\\
2\cos^2(\pi/2n), & n\equiv 2\pmod 4,\\
2\cos^2(\pi/4n), & n\ \text{odd},
\end{cases}
\]
where the equality follows by the same monotonicity and the presence (when $2\mid n$) of the $j=n/2$ column with value $1\le$ the right-hand side. Combining the upper and lower bounds yields the asserted equality for $M(Q^{\rm (ph)})$ and proves optimality within the cycle family.

Finally, Lemma~\ref{lem:coh} with $M(Q^{\rm (ph)})$ as above gives the sufficient condition
\[
1+\sum_{j=2}^n\lambda_j>1-\tfrac{1}{M(Q^{\rm (ph)})},
\]
which is precisely the stated $\delta^{\rm (ph)}_n$.
\end{proof}

\begin{cor}[Uniform phase bound]\label{cor:uniform}
As $\delta_\ast:=\sup_{n\ge3}\delta^{\mathrm{(ph)}}_n=\tfrac12$, for every $n\ge3$ and every Sule\u{\i}manova list $(1,\lambda_2,\ldots,\lambda_n)$ with $\lambda_j\le0$,
\[
1+\sum_{j=2}^{n}\lambda_j \;>\; \delta_\ast\ (= \tfrac12)
\quad\Longrightarrow\quad
Q^{(\mathrm{ph})}\,\Lambda\,\big(Q^{(\mathrm{ph})}\big)^{\!\top}\ \text{is symmetric and doubly stochastic}.
\]
\end{cor}

\section{A coherence criterion}

While the phase-optimised basis successfully removed the `$8 \mid n$' artefact, its coherence still approaches the maximum value of 2 as $n \to \infty$, preventing a uniform improvement on the $\delta < 1/2$ bound. To achieve a stronger result, we must look beyond the eigenbasis of the cycle graph. This section develops a general tool to guide this search. We first isolate the key property of the basis that governs the SDIEP bound, what we call its \emph{coherence}, a measure of the peak magnitude of its entries, and prove that a lower coherence directly implies a better bound $\delta$. Armed with this criterion, we then exhibit two families of `flat' bases, constructed from Hadamard matrices, which possess the ideal coherence of $M(Q)=1$. For dimensions where such matrices do exist, this allows us to prove the realisability of \emph{all} Sule\u{\i}manova spectra with a strictly positive trace sum.

\begin{lem}[Generic coherence criterion]\label{lem:coh}
Let $Q$ be any (real) orthogonal matrix with first column $q_0=\1/\sqrt n$. Define
\[
M(Q):=\sup_{1\le j\le n-1}\ \sup_{0\le k,l\le n-1}\ n\,|q_j(k)q_j(l)|
\ =\ n\max_{1\le j\le n-1}\|q_j\|_\infty^2.
\]
If $1+\sum_{j=2}^n\lambda_j>1-\tfrac1{M(Q)}$ and $\lambda_j\le0$, then $P(\Lam)=Q\Lam Q^\top$ is symmetric and doubly stochastic.
\end{lem}

\begin{proof}
As in Lemma~\ref{lem:entrybound}, $P(\Lam)\1=\1$ and
\[
p_{kl}=\frac1n\Bigl(1+\sum_{j=1}^{n-1}\lambda_j\,n\,q_j(k)q_j(l)\Bigr)
\ \ge\ \frac1n\Bigl(1+M(Q)\sum_{j=1}^{n-1}\lambda_j\Bigr)
\]
since $\lambda_j\le0$ and $n\,q_j(k)q_j(l)\le M(Q)$.
\end{proof}

The cycle basis is the standard Fourier basis for the cyclic group $\mathbb{Z}_n$. To achieve the minimal coherence of $M=1$, we turn from the cycle graph to a different algebraic structure: the \emph{abelian 2-group}. An abelian 2-group is a group that is commutative and in which every element has an order that is a power of two; for a finite group, this is equivalent to its size being a power of two. The canonical example for any dimension $n=2^d$ is the elementary abelian 2-group $G = (\mathbb{Z}/2\mathbb{Z})^d$, whose elements can be viewed as binary vectors $\mathbf{x}=(x_1, \dots, x_d)$ of length $d$, and whose operation is vector addition modulo $2$ (bitwise XOR).

The natural analogue of the Fourier basis on this group is the basis of its \emph{characters}. A character is a homomorphism $\chi\colon G \to \{\pm 1\}$. For each indexing vector $\mathbf{a} \in G$, there is a corresponding character $\chi_{\mathbf{a}}: G \to \{\pm 1\}$ defined by the dot product:
\[
\chi_{\mathbf{a}}(\mathbf{x}) = (-1)^{\mathbf{a}\cdot\mathbf{x}} = (-1)^{\sum_{i=1}^d a_i x_i}.
\]
The \emph{Sylvester} (or \emph{Walsh--Hadamard}) matrix $H_n$ is constructed by indexing its rows and columns by the elements of $G$ (typically in lexicographical order) and setting the entry at row $\mathbf{a}$ and column $\mathbf{x}$ to be $\chi_{\mathbf{a}}(\mathbf{x})$.
\

\begin{rem}[Hadamard matrix orders]
It is a classical open problem, known as the \emph{Hadamard conjecture}, that Hadamard matrices exist for every order $n$ equal to $1$, $2$, or a multiple of $4$ \cite{HedayatWallis}. Therefore, a Walsh--Hadamard matrix of order $3$ does not exist.
\end{rem}

A simple recursive construction, known as the Sylvester construction, generates these matrices for all orders $n=2^d$:
\[
H_1 = \begin{pmatrix} 1 \end{pmatrix}, \quad H_2 = \begin{pmatrix} 1 & 1 \\ 1 & -1 \end{pmatrix}, \quad H_{2n} = \begin{pmatrix} H_n & H_n \\ H_n & -H_n \end{pmatrix}.
\]
For example, the matrix for $n=4=2^2$ is constructed as follows:
\[
H_4 = \begin{pmatrix} H_2 & H_2 \\ H_2 & -H_2 \end{pmatrix} = 
\left(\begin{array}{cc|cc}
1 & 1 & 1 & 1 \\
1 & -1 & 1 & -1 \\\hline
1 & 1 & -1 & -1 \\
1 & -1 & -1 & 1
\end{array}\right).
\]
By construction, the entries of $H_n$ are all $\pm 1$ and its columns are orthogonal. When normalised to $Q = H_n/\sqrt{n}$, we obtain an orthogonal matrix whose entries are all $\pm 1/\sqrt{n}$. This ``flat'' structure gives a coherence of $M(Q)=1$, providing a powerful tool for the SDIEP as formalised in the following theorem.

\begin{thm}[bases stemming from abelian $2$-groups]\label{thm:hypercube}
If $n=2^d$, let $H_n$ be the Sylvester (Walsh--Hadamard) matrix of order $n$. After sign/permutation normalisation, $Q=H_n/\sqrt n$ is orthogonal with first column $\1/\sqrt n$ and entries $\pm 1/\sqrt n$, hence $M(Q)=1$. Therefore, for \emph{every} Sule\u{\i}manova spectrum with $1+\sum_{j=2}^n\lambda_j>0$, the matrix $Q\Lam Q^\top$ is symmetric and doubly stochastic.
\end{thm}

\begin{thm}[Hadamard orders via Paley-type constructions]\label{thm:hadamard}
Whenever a (real) Hadamard matrix $H$ of order $n$ exists with first column $\1$, taking $Q=H/\sqrt n$ yields $M(Q)=1$. Hence every Sule\u{\i}manova spectrum with $1+\sum_{j=2}^n\lambda_j>0$ is realised by a symmetric doubly stochastic matrix via $Q\Lam Q^\top$.
\end{thm}

\begin{proof}[Proof of Theorems \ref{thm:hypercube} and \ref{thm:hadamard} (combined)]
Let $H$ be a real Hadamard matrix of order $n$. By multiplying columns by $\pm1$ and permuting them (operations preserving the Hadamard property), normalise so that the first column is $\1$. Set $Q:=H/\sqrt n$. Then $Q$ is real orthogonal, its first column is $q_0=\1/\sqrt n$, and all entries satisfy $q_j(k)\in\{\pm 1/\sqrt n\}$.

Hence $M(Q)=\sup_{j\ge1,k,l}n|q_j(k)q_j(l)|=1$. By Lemma~\ref{lem:coh}, whenever $1+\sum_{j=2}^n\lambda_j>1-\tfrac1{M(Q)}=0$ and $\lambda_j\le 0$, the matrix $P(\Lambda)=Q\Lambda Q^\top$ is entrywise nonnegative. Symmetry is immediate, and
\[
P(\Lambda)\,\1=Q\Lambda Q^\top\1=Q\Lambda(\sqrt n\,e_0)=\sqrt n\,Qe_0=\1
\]
so $P(\Lambda)$ is doubly stochastic. The Walsh--Hadamard case corresponds to $n=2^d$; Paley-type constructions and other Hadamard orders give further instances.
\end{proof}

\begin{rem}
(i) The result requires the \emph{strict} condition $1+\sum\lambda_j>0$. Our lower bound on the matrix entries is $p_{kl}\ge \tfrac{1}{n}(1+\sum\lambda_j)$, which guarantees positivity under the strict inequality. The boundary case $1+\sum\lambda_j=0$ is not guaranteed to be realisable and is known to fail for certain spectra (e.g., $(1,0,\dots,0,-1)$ in odd dimensions).

(ii) Our use of Hadamard bases provides a constructive proof for \emph{all} (not necessarily sorted) Sule\u{\i}manova lists in these dimensions, provided the trace sum is positive. This complements results such as that of Johnson and Paparella \cite{JP16}, which establish existence for \emph{normalised} (sorted) Sule\u{\i}manova lists.
\end{rem}

\section{Further remarks and extensions}

\subsection*{Asymptotics (large \texorpdfstring{$n$}{n})}
For small angles, $\cos^2\theta=1-\theta^2+O(\theta^4)$, hence
\[
1-\frac{1}{2\cos^2\theta}\ =\ \frac12\ -\ \frac{\theta^2}{2}\ +\ O(\theta^4).
\]
For the canonical cycle basis, $\theta=\Delta_n=\frac{\pi}{4n}\rho(n)$ (Lemma~\ref{lem:cn}), so
\[
\delta_n\ =\
\begin{cases}
\displaystyle \frac12, & n\equiv 0\pmod{8},\\[1mm]
\displaystyle \frac12 - \frac{\pi^2}{32n^2}\rho(n)^2\ +\ O(n^{-4}), & n\not\equiv 0\pmod{8}.
\end{cases}
\]
Explicitly: $\delta_n=\frac12-\frac{\pi^2}{2n^2}+O(n^{-4})$ for $n\equiv 4\pmod{8}$; $\delta_n=\frac12-\frac{\pi^2}{8n^2}+O(n^{-4})$ for $n\equiv 2,6\pmod{8}$; and $\delta_n=\frac12-\frac{\pi^2}{32n^2}+O(n^{-4})$ for odd $n$. For the phase-optimised basis (Theorem~\ref{thm:phase-correct}),
\[
\delta_n^{\rm (ph)}=
\begin{cases}
\displaystyle \frac12-\frac{\pi^2}{2n^2}+O(n^{-4}),   & n\equiv 0\pmod{4},\\[2mm]
\displaystyle \frac12-\frac{\pi^2}{8n^2}+O(n^{-4}),  & n\equiv 2\pmod{4},\\[2mm]
\displaystyle \frac12-\frac{\pi^2}{32n^2}+O(n^{-4}), & n\ \text{odd}.
\end{cases}
\]

\subsection*{Tensor products and direct products of graphs}
If $Q^{(1)}$ and $Q^{(2)}$ are orthogonal with first columns $\1/\sqrt{n_1}$ and $\1/\sqrt{n_2}$, then $Q:=Q^{(1)}\otimes Q^{(2)}$ is orthogonal with first column $\1/\sqrt{n_1n_2}$ and
\[
M(Q)\ \le\ M\!\big(Q^{(1)}\big)\,M\!\big(Q^{(2)}\big).
\]
Thus Lemma~\ref{lem:coh} gives the sufficient sum-only condition
\[
1+\sum_{j=2}^{n_1n_2}\lambda_j\ >\ 1-\frac{1}{M(Q^{(1)})M(Q^{(2)})}
\]
for entrywise nonnegativity via $Q\Lam Q^\top$. In particular, flat ($M=1$) factors propagate.

\section{Limitations and open problems}

For the canonical cycle basis, $M(Q)=2$ exactly when $8\mid n$ (Lemma~\ref{lem:cn}), so no sum-only improvement below $\tfrac12$ is possible there. Even when $8\nmid n$, $M(Q)=2\cos^2\Delta_n=2\bigl(1-O(n^{-2})\bigr)$, so the gain is $O(n^{-2})$ (Theorem~\ref{thm:cycle}). The phase-optimised basis yields $\delta_n^{\rm (ph)}<\tfrac12$ for all $n$, but still not a single \emph{uniform} $\delta<\tfrac12$ independent of $n$.

\medskip
\noindent\emph{Open problem.} Is it possible to construct, for each $n$, a real orthogonal $Q$ with $q_0=\1/\sqrt n$ and $M(Q)\le 2-\varepsilon$ for some $\varepsilon>0$ independent of $n$. By Lemma~\ref{lem:coh}, this would give a universal $\delta<\tfrac12$ for all (not necessarily normalised) Sule\u{\i}manova lists.

\appendix

\section*{Appendix: Values and visualisation of the phase-optimised bound}

Using Theorem~\ref{thm:phase-correct}, concrete values (rounded) are:

\begin{center}
\begin{tabular}{r r @{\qquad} r r @{\qquad} r r @{\qquad} r r}
\hline
$n$ & $\delta^{\mathrm{(ph)}}_n$ & $n$ & $\delta^{\mathrm{(ph)}}_n$ & $n$ & $\delta^{\mathrm{(ph)}}_n$ & $n$ & $\delta^{\mathrm{(ph)}}_n$ \\
\hline
3  & 0.4641016151 & 4  & 0                & 5  & 0.4874571845 & 6  & 0.4641016151 \\
7  & 0.4936524043 & 8  & 0.4142135624     & 9  & 0.4961728669 & 10 & 0.4874571845 \\
12 & 0.4641016151 & 16 & 0.4802169351     & 20 & 0.4874571845 & 24 & 0.4913338099 \\
30 & 0.4986267123 & 40 & 0.4969030207     & 60 & 0.4986267123 & 100 & 0.4995061949 \\
\hline
\end{tabular}
\end{center}

\begin{figure}[t]
\centering
\begin{tikzpicture}
\begin{axis}[
  width=0.9\textwidth,
  height=7cm,
  xlabel={$n$},
  ylabel={$\,\delta^{\mathrm{(ph)}}_n$},
  xmin=3, xmax=60,
  ymin=0, ymax=0.52,
  grid=both,
  legend pos=south east,
  ytick distance=0.1,
  xtick distance=5
]
\addplot+[no marks, thick, domain=4:60, samples=300] {1 - 1/(2*(cos(180/x))^2)};
\addlegendentry{$n\equiv 0\pmod{4}$: $1 - \dfrac{1}{2\cos^2(\pi/n)}$}
\addplot+[no marks, thick, domain=6:60, samples=300, dashed] {1 - 1/(2*(cos(90/x))^2)};
\addlegendentry{$n\equiv 2\pmod{4}$: $1 - \dfrac{1}{2\cos^2(\pi/(2n))}$}
\addplot+[no marks, thick, domain=3:60, samples=300, densely dotted] {1 - 1/(2*(cos(45/x))^2)};
\addlegendentry{$n$ odd: $1 - \dfrac{1}{2\cos^2(\pi/(4n))}$}
\addplot+[only marks] coordinates {(3,0.4641016151) (4,0) (5,0.4874571845) (6,0.4641016151) (7,0.4936524043)
(8,0.4142135624) (9,0.4961728669) (10,0.4874571845) (12,0.4641016151) (16,0.4802169351) (20,0.4874571845) (24,0.4913338099)};
\addlegendentry{sample $n$}
\end{axis}
\end{tikzpicture}
\caption{Phase-optimised threshold $\delta^{\mathrm{(ph)}}_n$ by congruence class. Curves are envelopes; only integer $n$ on the labelled classes apply.}
\end{figure}
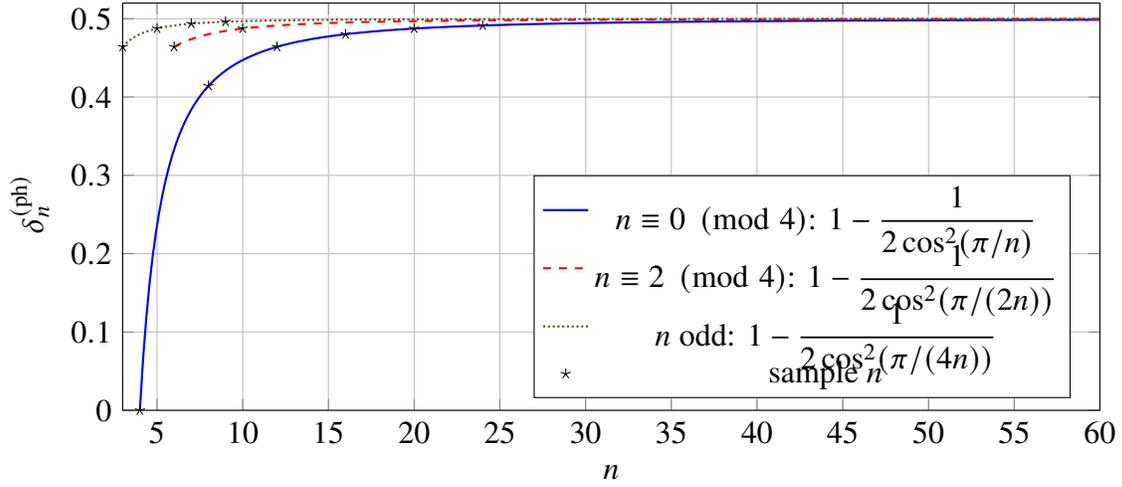

\end{document}